\pdfoutput=1 
\documentclass[11pt]{amsart}
\usepackage[utf8]{inputenc}
\usepackage{amssymb,amsmath,enumerate,amsthm,enumitem,colonequals,mlmodern, tikz-cd, microtype}
\usepackage{microtype}
\usepackage[mathcal]{eucal}
\usepackage[top=3.75cm, bottom=3cm, left=3.5cm, right=3.5cm]{geometry}

\makeatletter
\@namedef{subjclassname@2020}{\textup{2020} Mathematics Subject Classification}
\makeatother

\usepackage{xcolor}
\colorlet{darkblue}{blue!55!black}
\colorlet{darkcyan}{cyan!50!black}
\colorlet{darkgreen}{green!60!black}

\usepackage{hyperref}
\hypersetup{
    colorlinks=true,
    linkcolor= darkblue,
    urlcolor= darkcyan,
    citecolor= darkgreen,
}

\def\eqref#1{\textcolor{darkblue}{(\ref{#1})}}

\PassOptionsToPackage{hyphens}{url}\usepackage{hyperref}

\usepackage[nameinlink]{cleveref} 
\Crefformat{section}{#2\S#1#3} 
\Crefmultiformat{section}{#2\S\S#1#3}{ and~#2#1#3}{, #2#1#3}{, and~#2#1#3}

\usepackage[pagewise]{lineno}
\let\oldequation\equation
\let\oldendequation\endequation
\renewenvironment{equation}{\linenomathNonumbers\oldequation}{\oldendequation\endlinenomath}
\expandafter\let\expandafter\oldequationstar\csname equation*\endcsname
\expandafter\let\expandafter\oldendequationstar\csname endequation*\endcsname
\renewenvironment{equation*}{\linenomathNonumbers\oldequationstar}{\oldendequationstar\endlinenomath}
\let\oldalign\align
\let\oldendalign\endalign
\renewenvironment{align}{\linenomathNonumbers\oldalign}{\oldendalign\endlinenomath}
\expandafter\let\expandafter\oldalignstar\csname align*\endcsname
\expandafter\let\expandafter\oldendalignstar\csname endalign*\endcsname
\renewenvironment{align*}{\linenomathNonumbers\oldalignstar}{\oldendalignstar\endlinenomath}

\usepackage{tikz-cd}




\newcounter{intro}
\newcounter{result}

\theoremstyle{definition}
\newtheorem{theorem}[result]{Theorem}
\newtheorem{lemma}[result]{Lemma}
\newtheorem{proposition}[result]{Proposition}
\newtheorem{corollary}[result]{Corollary}
\newtheorem{definition}[result]{Definition}
\newtheorem{example}[result]{Example}
\newtheorem{remark}[result]{Remark}

\newtheorem{notation}[result]{Notation}

\newtheorem{setup}[result]{Setup}

\newtheorem{disclaimer}[result]{Disclaimer}

\newtheorem*{ack}{Acknowledgements}

\numberwithin{equation}{section}
\numberwithin{result}{section}

\title[Approximation and Rouquier dimension]{Approximation by perfect complexes detects Rouquier dimension}

\author[P.~Lank]{Pat Lank}
\address{P.~Lank,
Department of Mathematics,
University of South Carolina, 
Columbia, SC 29208,
U.S.A.}
\email{plankmathematics@gmail.com}

\author[N.~Olander]{Noah Olander}
\address{N.~Olander,
Korteweg-de Vries Institute for Mathematics,
University of Amsterdam, 
Science Park 105-107, 1098 XG,
Amsterdam, Netherlands}
\email{n.b.olander@uva.nl}

\date{\today}



\begin{document}

\begin{abstract}
    This work explores bounds on the Rouquier dimension in the bounded
    derived category of coherent sheaves on Noetherian schemes. 
    By utilizing approximations, we exhibit that Rouquier dimension is
    inherently characterized by the number of cones required to 
    build all perfect complexes. We use this to prove sharper bounds on Rouquier dimension
    of singular schemes.
    Firstly, we show Rouquier dimension doesn't go up along \'{e}tale extensions and is invariant under \'{e}tale covers of affine schemes admitting a dualizing complex. Secondly, we demonstrate that the Rouquier dimension of the
    bounded derived category for a curve, with a delta
    invariant of at most one at closed points, is no larger than two. Thirdly,
    we bound the Rouquier dimension for the bounded derived category of a
    (birational) derived splinter variety by that of a resolution of singularities.
\end{abstract}

\maketitle

\section{Introduction}
\label{sec:intro}

Our work establishes a technical result concerning retracts of good approximations in triangulated categories with a bounded $t$-structure, see \Cref{lem:summandofapprox}. This result provides a valuable perspective for proving sharp bounds on the Rouquier dimension of triangulated categories of interest in algebraic geometry.

The concept of generation in a triangulated category $\mathcal{T}$ was initially introduced in \cite{BVdB:2003}. For an object $G$ of $\mathcal{T}$, we denote $\langle G \rangle_n$ as the smallest full subcategory of $\mathcal{T}$ generated from $G$ using finite coproducts, shifts, retracts, and at most $n$ cones. The \textit{Rouquier dimension} of $\mathcal{T}$ is the minimal integer $n$ such that there exists an object $G$ of $\mathcal{T}$ satisfying $\langle G \rangle_n = \mathcal{T}$; any object such that $\langle G \rangle_n = \mathcal{T}$ for some $n$ is called a \textit{strong generator}. More generally, an object $G$ is called a \textit{classical generator} if $\bigcup^{\infty}_{n=0} \langle G \rangle_n = \mathcal{T}$. For additional background, refer to \Cref{def:strong_generators}. 

Consider a Noetherian scheme $X$, and let $D^b_{\operatorname{coh}}(X)$ denote the derived category of bounded complexes with coherent cohomology. There has been progress made regarding when $D^b_{\operatorname{coh}}(X)$ admits a strong or classical generators \cite{Neeman:2021, Aoki:2021, Jatoba:2021,DeDeyn/Lank/ManaliRahul:2024a,DeDeyn/Lank/ManaliRahul:2024b}, as well as on how to identify such objects explicitly \cite{BILMP:2023, Olander:2023, Rouquier:2008, Orlov:2009, Lank:2023}. The applications of generation for $D^b_{\operatorname{coh}}(X)$ have geometrically flavored consequences too. This includes detecting openness of the regular locus for $X$ \cite{Dey/Lank:2024a, Iyengar/Takahashi:2019}, as well as characterizations for singularities arising in birational geometry \cite{Lank/Venkatesh:2024, Lank/McDonald/Venkatesh:2025}.

For a smooth variety $X$ over a field, the Rouquier dimension of $D^b_{\operatorname{coh}}(X)$ is bounded above by $2 \dim X$ \cite{Rouquier:2008, Olander:2022}. Initially formulated for smooth quasi-projective varieties in \cite{Orlov:2009}, it is conjectured that the Rouquier dimension of $D^b_{\operatorname{coh}}(X)$ for a smooth variety $X$ coincides with $\dim X$. This expectation has been verified in various instances \cite{Rouquier:2008, Ballard/Favero:2012, Orlov:2009, Ballard/Favero/Katzarkov:2014, Hanlon/Hicks/Lazarev:2023, Pirozhkov:2023}.

Calculating the Rouquier dimension of $D^b_{\operatorname{coh}}(X)$ can be challenging. However, when certain mild assumptions on $X$ hold, powerful tools become available. For example, if $X$ is a projective scheme over a commutative Noetherian ring, the converse coghost lemma \cite{OS:2012} offers a valuable method. This lemma facilitates the translation of questions about generation into concerns regarding the vanishing of compositions of specific maps in the bounded derived category. 

In \cite{OS:2012}, the converse coghost lemma was used to show that for any projective scheme $X$ over a commutative Noetherian ring, any strongly generated strictly full triangulated subcategory of $D^b_{\operatorname{coh}}(X)$ that contains every perfect complex must coincide with $D^b_{\operatorname{coh}}(X)$, see  \cite[Theorem 2]{OS:2012}. We extend this here to arbitrary Noetherian schemes and additionally prove that if every perfect complex can be built in at most $d$ cones from a given object, then so can $D^b_{\operatorname{coh}}(X)$. As shown in \cite{Lipman/Neeman:2007}, every object $E$ in $D^b_{\operatorname{coh}}(X)$ has a good approximation by objects in the subcategory $\operatorname{perf}X$, which consists of perfect complexes on $X$. The subsequent result is our first application of \Cref{lem:summandofapprox}.

\begin{theorem}\label{thm:generatingperf}
    Suppose $X$ is a Noetherian scheme. 
    Let $G$ be an object of $D^b_{\operatorname{coh}}(X)$ and $d$ a non-negative integer. The following are equivalent:
    \begin{enumerate}
        \item $D^b_{\operatorname{coh}}(X) = \langle G \rangle_d$
        \item $\operatorname{perf} X \subseteq \langle G \rangle_d$. 
    \end{enumerate}
\end{theorem}

\Cref{thm:generatingperf} serves as a special case of a broader statement that we establish later regarding triangulated categories with a bounded $t$-structure; refer to \Cref{prop:approximation_bounded_t_structure} for details. For this reason, our result holds more generally for Noetherian algebraic spaces or left/right Noetherian non-commutative rings since approximation by perfect complexes also holds in these contexts. Our proof strategy crucially doesn't impose any $\operatorname{hom}$-finiteness assumptions on the triangulated category as was required in \cite[Theorem 2]{OS:2012}, which allows the extension from projective schemes to more general ones. \Cref{thm:generatingperf} can be applied to produce good bounds on Rouquier dimension of singular varieties. For instance, we use it to reprove \cite[Corollary 7.2]{Burban/Drozd:2017} which says that the bounded derived category of a curve with only nodes and cusps as singularities has Rouquier dimension $\leq 2$, see \Cref{cor:sharp_bound_rouquier_dimension_curves}. In what follows is a different flavor of applications of \Cref{lem:summandofapprox}.

\begin{theorem}\label{thm:summandofpullpush}
    Let $X$ be a Noetherian scheme. 
    \begin{enumerate}
        \item If $f \colon  U\to X$ is a quasi-compact separated \'{e}tale morphism\footnote{These hypotheses will ensure that $U$ is Noetherian.}, then every object of $D^b_{\operatorname{coh}}(U)$ is a direct summand of $\mathbb{L} f^\ast$ of an object of $D^b_{\operatorname{coh}}(X)$. In particular, one has
        \begin{displaymath}
            \dim D^b_{\operatorname{coh}}(U) \leq \dim D^b_{\operatorname{coh}}(X).
        \end{displaymath}
        \item If $f  \colon  Y \to X$ is proper and the natural map $\mathcal{O}_X \to\mathbb{R} f_\ast \mathcal{O}_Y$ splits in $D^b_{\operatorname{coh}}(X)$, then every object $A$ of $D^b_{\operatorname{coh}}(X)$ is a direct summand of $\mathbb{R} f_\ast E$ for $E$ an object of $D^b_{\operatorname{coh}}(Y)$. In particular, one has
        \begin{displaymath}
            \dim D^b_{\operatorname{coh}}(X) \leq \dim D^b_{\operatorname{coh}}(Y).
        \end{displaymath}
    \end{enumerate}
\end{theorem}

Part (1) was known with $D^b_{\operatorname{coh}}$ replaced by $\operatorname{perf}$, see \cite[Theorems 3.5 and 4.2]{Balmer:2016}. In the special case that $\mathcal{O}_X \to\mathbb{R} f_\ast \mathcal{O}_Y$ is an isomorphism, part (2) follows from an argument of Kawamata \cite[Lemma 7.4]{Kawamata:2022}. A consequence of \Cref{thm:summandofpullpush} is demonstrated in \Cref{ex:derived_splinters}, where the Rouquier dimension of $D^b_{\operatorname{coh}}(X)$ is shown to be no more than twice the Krull dimension of $X$ for (birational) derived splinter varieties over a field. Furthermore, under the assumption that \cite[Conjecture 10]{Orlov:2009} holds, the Rouquier dimension of the bounded derived category would become the Krull dimension for such varieties. This observation was previously known only for varieties with rational singularities.

We are able to prove a partial reverse inequality to the one in \Cref{thm:summandofpullpush}(1):

\begin{theorem}\label{thm:etale_descent_rouquier_dimension}
    If $f \colon U \to X$ is a surjective \'{e}tale morphism of Noetherian affine schemes and $X$ admits a dualizing complex, then $\dim D^b_{\operatorname{coh}}(U) = \dim D^b_{\operatorname{coh}}(X)$. 
\end{theorem}

We prove this using a blend of algebraic and geometric arguments, namely Letz's converse ghost Lemma and Deligne's formula, see \Cref{rmk:coghost_lemma} for further background. As applications, we prove that the Rouquier dimension of the bounded derived category of a Noetherian local ring with a dualizing complex is invariant under (strict) Henselization (\Cref{cor:hensel}), and that the bounded derived category of an affine nodal curve is 1 (\Cref{cor:nodal}).  
\begin{ack}
    The authors would like to thank Amnon Neeman, Srikanth B. Iyengar, Josh Pollitz, Souvik Dey, and Jan {\v{S}}\v{t}ov\'{i}{\v{c}}ek for taking a look at an earlier draft of our work. The second author was funded by the Deutsche Forschungsgemeinschaft (DFG, German Research Foundation) under Germany's Excellence Strategy – EXC-2047/1 – 390685813 and by the European Union (ERC, ZETA-FM, 8641450). He is grateful to the Hausdorff Research Institute for Mathematics for its support and hospitality while this work was being carried out. The recently posted paper \cite{Biswas/Chen/Rahul/Parker/Zheng:2023} on approximability and bounded $t$-structures touches on similar subject matter but proves a different flavor of results. It would be interesting to find a relation between their work and ours. Both authors thank the anonymous referee for help suggestions on improvements to our work. 
\end{ack}

\begin{disclaimer}
    We have chosen to write primarily in the language of Noetherian schemes and the bounded derived category of coherent sheaves for the sake of simplicity. However, many results in our work extend to quasi-compact quasi-separated algebraic spaces and their bounded pseudo-coherent derived categories.
\end{disclaimer}

\begin{notation}
    Let $X$ be a scheme.
    \begin{itemize}
        \item $D(X)$ denotes the unbounded derived category of the category of $\mathcal{O}_X$-modules 
        \item $D_{\operatorname{Qcoh}}(X)$ denotes the strictly full subcategory of objects in $D(X)$ whose cohomology sheaves are quasi-coherent $\mathcal{O}_X$-modules
        \item $\operatorname{perf}X$ denotes the strictly full subcategory of $D(X)$ consisting of perfect complexes
         \item If $X$ is locally Noetherian, $D^b_{\operatorname{coh}}(X)$ (resp. $D^-_{\operatorname{coh}}(X)$) denotes the strictly full subcategory of $D(X)$ whose objects are bounded (resp. bounded above) with coherent cohomology sheaves.
         \item We will occasionally abuse notation and write $D^b_{\operatorname{coh}}(R)$ instead of $D^b_{\operatorname{coh}}(\operatorname{Spec}(R))$ for a Noetherian ring $R$, which coincides with the bounded derived category of finitely generated $R$-modules.
    \end{itemize}
\end{notation}

\section{Generation}
\label{sec:generation}

Before proceeding forth to our results, we provide a quick overview for generation in triangulated categories and draw content from \cite{Krause:2022, Huybrechts:2006, BVdB:2003, Rouquier:2008, Neeman:2021, OS:2012}. Consider a triangulated category $\mathcal{T}$ equipped with the shift functor $[1]\colon \mathcal{T} \to \mathcal{T}$.
 
\begin{definition}\label{def:thick_subcategory}
    Let $\mathcal{S}$ be a subcategory of $\mathcal{T}$.
        \begin{enumerate}
            \item $\operatorname{add}(\mathcal{S})$ is the smallest strictly full subcategory of $\mathcal{T}$ containing $\mathcal{S}$ closed under shifts, finite coproducts, and retracts
            \item $\mathcal{S}$ is said to be \textbf{thick} if it is closed under retracts
            \item $\langle \mathcal{S} \rangle$ denotes the smallest strictly full, thick, triangulated  subcategory in $\mathcal{T}$ containing $\mathcal{S}$
            \item If $\mathcal{T}$ admits all small coproducts, then $\operatorname{Add}(\mathcal{S})$ is the smallest strictly full subcategory of $\mathcal{T}$ containing $\mathcal{S}$ closed under shifts, small coproducts, and retracts.
        \end{enumerate}
\end{definition}

\begin{remark}
    Let $\mathcal{S}$ be a subcategory of $\mathcal{T}$. There exists an exhaustive filtration:
    \begin{displaymath}
        \langle \mathcal{S} \rangle_0 \subseteq \langle \mathcal{S} \rangle_1 \subseteq \cdots \subseteq \bigcup_{n=0}^\infty \langle \mathcal{S} \rangle_n = \langle \mathcal{S} \rangle.
    \end{displaymath}
    where the additive subcategories are defined as
    \begin{enumerate}
        \item $\langle \mathcal{S} \rangle_0$ is the full subcategory consisting of all objects isomorphic to the zero object
        \item $\langle \mathcal{S} \rangle_1:= \operatorname{add}(\mathcal{S})$
        \item $\langle \mathcal{S} \rangle_n := \operatorname{add} \big( \{ \operatorname{cone}(\varphi) : \varphi\in \operatorname{Hom}_\mathcal{T} (\langle \mathcal{S} \rangle_{n-1}, \langle \mathcal{S} \rangle_1)  \} \big)$.
    \end{enumerate}
    The specific presentation for $(3)$ occurs in \cite{Krause:2023}.
\end{remark}

\begin{definition}\label{def:coprod_subcategory}(\cite[$\S 1$]{Neeman:2021})
    Let $\mathcal{S}$ be a subcategory of $\mathcal{T}$.
    \begin{enumerate}
        \item $\operatorname{coprod}_0 \big(\mathcal{S} ( -\infty,\infty)\big)$ is the full subcategory consisting of all objects isomorphic to the zero object
        \item $\operatorname{coprod}_1 \big(\mathcal{S} ( -\infty,\infty)\big)$ is the smallest strictly full subcategory containing all finite coproducts of objects of the form $S[i]$ for some integer $i$ and $S \in \mathcal{S}$
        \item If $n\geq 2$, then $\operatorname{coprod}_n \big(\mathcal{S} ( -\infty,\infty)\big)$ is defined as the collection of objects of the form $\operatorname{cone}(\varphi)$ where $\varphi\colon A \to B$ such that $A\in \operatorname{coprod}_{n-1} \big(\mathcal{S} ( -\infty,\infty)\big)$ and $B\in \operatorname{coprod}_1 \big(\mathcal{S} ( -\infty,\infty) \big) \big)$.
    \end{enumerate}
\end{definition}

\begin{remark}
\label{rmk:retractsofcoprod}
    There is a distinction between \Cref{def:thick_subcategory} and \Cref{def:coprod_subcategory}. In the former, one utilizes cones, shifts, and \textit{retracts} of finite coproducts of an object. In contrast, the latter involves the use of cones, shifts, and finite coproducts. Moreover, for any $n \geq 0$, it can be verified that $\operatorname{coprod}_n (\mathcal{S})$ forms a strictly full subcategory of $\langle \mathcal{S} \rangle_n$. More precisely, $\langle \mathcal{S} \rangle _n$ consists exactly of retracts of objects of $\operatorname{coprod}_n (\mathcal{S}(-\infty, \infty)$. Also, an equivalent definition of $\langle \mathcal{S} \rangle_n$ is 
    $\operatorname{add} \big( \{ \operatorname{cone}(\varphi) : \varphi\in \operatorname{Hom}_\mathcal{T} (\langle \mathcal{S} \rangle_{1}, \langle \mathcal{S} \rangle_{n-1})  \} \big)$ 
    and similarly for $\operatorname{coprod}_n \big(\mathcal{S} (  -\infty,\infty)\big)$. See \cite[Section 2.2]{BVdB:2003}. It is important to note that $\operatorname{add}(\mathcal{S})$ in our notation allows direct summands, whereas $\operatorname{add}(\mathcal{S})$ in \cite{Neeman:2021, Aoki:2021} does not allow for direct summands.
\end{remark}

\begin{definition}\label{def:strong_generators}(\cite{ABIM:2010}, \cite{Rouquier:2008})
    Let $E,G$ be objects in $\mathcal{T}$.
    \begin{enumerate}
        \item If $E$ belongs to $\langle G \rangle$, then we say $E$ is \textbf{finitely built} by $G$. The \textbf{level}, denoted $\operatorname{level}_\mathcal{T}^G(E)$, of $E$ with respect to $G$ is the minimal non-negative integer $n$ such that $E\in \langle G \rangle_n$.
        \item If $\langle G \rangle =
        \mathcal{T}$, then $G$ is said to be a \textbf{classical generator} for $\mathcal{T}$.
        \item If $\langle G \rangle_n =
        \mathcal{T}$ for some $n\geq 0$, then $G$ is said to be a \textbf{strong generator} for $\mathcal{T}$. The \textbf{generation time} of $G$ is the minimal $n$ such that for all $E\in \mathcal{T}$ one has $\operatorname{level}_\mathcal{T}^G (E)\leq n+1$. 
        \item The \textbf{Rouquier dimension} of $\mathcal{T}$, is the smallest integer $d$ such that $\langle G \rangle_{d+1} = \mathcal{T}$ for some $G\in \mathcal{T}$, denoted $\dim \mathcal{T}$.
    \end{enumerate}
\end{definition}

\begin{example}\label{ex:strong_generator_examples}
    \hfill
    \begin{enumerate}
        \item \cite[Proposition 7.9]{Rouquier:2008} Let $X$ be a smooth quasi-projective variety over a field. If $\mathcal{L}$ is a very ample line bundle on $X$, then $\bigoplus^{\dim X}_{i=0} \mathcal{L}^{\otimes i}$ is a strong generator for $D^b_{\operatorname{coh}}(X)$ whose generation time is at most $2 \dim X$.
        \item \cite[Example 4.3]{Lank:2023} Let $\pi \colon \widetilde{X}\to X$ be a proper birational morphism of projective varieties over a field where $\widetilde{X}$ is smooth. If $\mathcal{L}$ is a very ample line bundle on $\widetilde{X}$, then $\mathbb{R}\pi_\ast (\bigoplus^{\dim X}_{i=0} \mathcal{L}^{\otimes i})$ is a strong generator for $D^b_{\operatorname{coh}}(X)$.
        \item \cite[Corollary 3.9]{BILMP:2023} Let $X$ be a quasi-projective variety over a perfect field of nonzero characteristic. Given a very ample line bundle $\mathcal{L}$ on $X$, $F_\ast^e (\bigoplus^{\dim X}_{i=0} \mathcal{L}^{\otimes i})$ is a strong generator for $D^b_{\operatorname{coh}}(X)$ where $e \gg 0$ and $F \colon X \to X$ is the Frobenius morphism.
        \item \cite[Corollary 5]{Olander:2023} Suppose $X$ is a quasi-affine Noetherian regular scheme of finite dimension. Then $\mathcal{O}_X$ is a strong generator for $D^b_{\operatorname{coh}}(X)$ whose generation time is at most $\dim X$.
    \end{enumerate}
\end{example}

\begin{definition}\label{def:big_generation}(\cite[$\S 2.2$]{BVdB:2003}, \cite[$\S 1$]{Neeman:2021}, \cite[$\S 3.3$]{Rouquier:2008})
    Let $\mathcal{S}$ be a subcategory of $\mathcal{T}$, and assume that $\mathcal{T}$ admits all small coproducts.
    \begin{enumerate}
        \item $\overline{\langle \mathcal{S} \rangle}_0$ is the full subcategory consisting of all objects isomorphic to the zero object
        \item $\overline{\langle \mathcal{S} \rangle}_1:= \operatorname{Add}(\mathcal{S})$
        \item $\overline{\langle \mathcal{S} \rangle}_n := \operatorname{Add} \big( \{ \operatorname{cone}(\varphi) : \varphi\in \operatorname{Hom}_\mathcal{T} (\overline{\langle \mathcal{S} \rangle}_{n-1}, \overline{\langle \mathcal{S} \rangle}_1)  \} \big)$.
    \end{enumerate}
    An object $G$ of $\mathcal{T}$ is said to be a \textbf{strong $\oplus$-generator} if there exists an $n \geq 0$ such that $\overline{\langle G \rangle}_n = \mathcal{T}$.
\end{definition}

\begin{remark}\label{rmk:big_to_small_generation_a_la_neeman}
    The special case where $\mathcal{T}=D_{\operatorname{Qcoh}}(X)$ for $X$ a Noetherian scheme is of primary interest to our work where \Cref{def:thick_subcategory} and \Cref{def:big_generation} interact. In particular, if $G$ is an object of $D^b_{\operatorname{coh}}(X)$ and $n\geq 0$, then 
    \begin{displaymath}
        \overline{\langle G \rangle}_n \cap D^b_{\operatorname{coh}}(X) = \langle G \rangle_n.
    \end{displaymath}
    This is readily verified from the proofs of Lemma 2.4-2.6 in \cite{Neeman:2021}. The special case where $G$ is in $\operatorname{perf}X$ is \cite[Proposition 2.2.4]{BVdB:2003}. Note that strong $\oplus$-generators with bounded coherent cohomology exist in $D_{\operatorname{Qcoh}}(X)$ when $X$ is a quasi-compact separated quasi-excellent scheme of finite Krull dimension, see \cite[Main Theorem]{Aoki:2021}. 
\end{remark}

\begin{definition}\label{def:coghost}(\cite{OS:2012,Kelly:1965,Beligiannis:2008})
    Let $\mathcal{C}$ be a full subcategory of $\mathcal{T}$. Choose objects $A,B$ in $\mathcal{T}$.
    \begin{enumerate}
        \item A map $g \colon A \to B$ is said to be $\mathcal{C}$-coghost if the induced map
        \begin{displaymath}
            \operatorname{Hom}_{\mathcal{T}} (B , C[n])\xrightarrow{(-) \circ g} \operatorname{Hom}_{\mathcal{T}} (A , C[n])
        \end{displaymath}
        vanishes for all objects $C$ in $\mathcal{C}$ and $n\in \mathbb{Z}$. An $n$-fold $\mathcal{C}$-coghost map is a composition of $n$ $\mathcal{C}$-coghost maps. The \textbf{coghost index} of $B$ with respect to $\mathcal{C}$, denoted $\operatorname{cogin}_{\mathcal{T}}^{\mathcal{C}} (B)$, is the smallest non-negative integer $n$ such that any $n$-fold $\mathcal{C}$-coghost map $A \to B$ vanishes.
        \item A map $g \colon A \to B$ is said to be $\mathcal{C}$-ghost if the induced map
        \begin{displaymath}
            \operatorname{Hom}_{\mathcal{T}} (C[n],A)\xrightarrow{g \circ (-)} \operatorname{Hom}_{\mathcal{T}} (C[n],B)
        \end{displaymath}
        vanishes for all objects $C$ in $\mathcal{C}$ and $n\in \mathbb{Z}$. An $n$-fold $\mathcal{C}$-ghost map is a composition of $n$ $\mathcal{C}$-ghost maps. The \textbf{ghost index} of $A$ with respect to $\mathcal{C}$, denoted $\operatorname{gin}_{\mathcal{T}}^{\mathcal{C}} (B)$, is the smallest non-negative integer $n$ such that any $n$-fold $\mathcal{C}$-ghost map $A \to B$ vanishes.
        \item If $\mathcal{C}$ consists of a single object $G$, then we write the coghost index as $\operatorname{cogin}^G_{\mathcal{T}}(-)$ and similarly for ghost index, and we refer to a $\mathcal{C}$-(co)ghost as a $G$-(co)ghost 
    \end{enumerate}
\end{definition}

\begin{remark}\label{rmk:coghost_lemma}
    Given a pair of objects $G,E$ in $\mathcal{T}$, the following is known \cite{Kelly:1965}:
    \begin{displaymath}
        \operatorname{cogin}^G_{\mathcal{T}}(E) , \operatorname{gin}^G_{\mathcal{T}}(E) \leq \operatorname{level}^G_{\mathcal{T}}(E).
    \end{displaymath}
    These inequalities are called the coghost and ghost lemmas. If $\mathcal{T}=D^b_{\operatorname{coh}}(X)$ where $X$ is a projective scheme over a commutative Noetherian ring, then $\operatorname{cogin}^G_{\mathcal{T}} (E) = \operatorname{level}^G_{\mathcal{T}} (E)$ by \cite[Theorem 4]{OS:2012}. This (more precisely, the inequality $\geq$) is sometimes called the converse coghost lemma. In the case $X$ admits a dualizing complex, one can deduce that also $\operatorname{gin}^G_{\mathcal{T}} (E)= \operatorname{level}^G_{\mathcal{T}} (E)$. See \cite[$\S 2.13$]{Letz:2021} for the special case when $X$ is an affine Noetherian scheme with a dualizing complex, which is all we need in this paper. We refer to this (more precisely, the inequality $\geq$) as the converse ghost lemma. 
\end{remark}

\begin{definition}\label{def:t_structure}(\cite{Beilinson/Bernstein/Deligne/Gabber:2018})
    A \textbf{$t$-structure} on $\mathcal{T}$ is a pair of strictly full subcategories $\mathcal{T}^{\leq 0},\mathcal{T}^{\geq 0}$ satisfying the following:
    \begin{enumerate}
        \item $\mathcal{T}^{\leq 0}[1]\subseteq \mathcal{T}^{\leq 0}$ and $\mathcal{T}^{\geq 0}[-1]\subseteq \mathcal{T}^{\geq 0}$
        \item $\operatorname{Hom}_{\mathcal{T}} (\mathcal{T}^{\leq 0} , \mathcal{T}^{\geq 0}[-1])=0$
        \item For any object $T$ in $\mathcal{T}$, there exists a distinguished triangle 
        \begin{displaymath}
            A \to T \to B[-1] \to A[1]
        \end{displaymath}
        where $A\in \mathcal{T}^{\leq 0}$ and $B\in \mathcal{T}^{\geq 0}$. 
    \end{enumerate}
    \end{definition}
    Define $\mathcal{T}^{\leq n} := \mathcal{T}^{\leq 0}[-n]$ and $\mathcal{T}^{\geq n}:= \mathcal{T}^{\geq 0}[-n]$. An object $E$ of $\mathcal{T}$ is \textbf{bounded} with respect to the $t$-structure if there is an integer $n > 0$ such that $E \in \mathcal{T}^{\leq n} \cap \mathcal{T}^{\geq -n}$. The $t$-structure is \textbf{bounded} provided every object is bounded. The \textbf{heart} of the $t$-structure is defined to be $\mathcal{T}^\heartsuit := \mathcal{T}^{\leq 0} \cap \mathcal{T}^{\geq 0}$.  The heart $\mathcal{T}^\heartsuit$ is a strictly full abelian subcategory of $\mathcal{T}$. Choose an integer $n$. The pair $(\mathcal{T}^{\leq n},\mathcal{T}^{\geq n})$ induces another $t$-structure on $\mathcal{T}$. The left (resp. right) adjoint of the inclusion $\mathcal{T}^{\geq n} \to \mathcal{T}$ (resp. $\mathcal{T}^{\leq n} \to \mathcal{T}$) is called a \textbf{truncation functor} of $\mathcal{T}$, and it is denoted $\tau_{\geq n} \colon \mathcal{T} \to \mathcal{T}^{\geq n}$ (resp. $\tau_{\leq n} \colon \mathcal{T} \to \mathcal{T}^{\leq n}$). Choose an object $E$ of $\mathcal{T}$. There exists a distinguished triangle
\begin{displaymath}
    \tau_{\leq n} (E) \to E \to \tau_{\geq n+1} (E) \to (\tau_{\leq n} (E))[1].
\end{displaymath}
These truncation functors give rise to the \textbf{$n$-cohomology functors} $H^n (-)\colon \mathcal{T} \to \mathcal{T}^\heartsuit$, which are defined as 
\begin{displaymath}
    H^n (E):= \tau_{\leq 0}\tau_{\geq 0} (E[n]).
\end{displaymath}
We will call a morphism $f \colon E \to F$ in $\mathcal{T}$ an \textbf{isomorphism in degree $n$} if one has $H^n(f)\colon H^n(E) \to H^n(F)$ is an isomorphism in $\mathcal{T}^\heartsuit$. The \textbf{amplitude} of a bounded object $E$ of $\mathcal{T}$ is $- \infty$ if the object is zero and 
\begin{displaymath}
    \sup\{ n \in \mathbb{Z} : H^n (E)\not=0\} - \inf\{ s \in \mathbb{Z} : H^s(E)\not=0 \}
\end{displaymath}
otherwise. In the latter case it is always an integer $\geq 0$. The $t$-structure is \textbf{non-degenerate} if $\cap_{n \geq 0} \mathcal{T}^{\geq n} = 0 = \cap_{n \geq 0} \mathcal{T}^{\leq - n}$. In this case, an object $E$ of $\mathcal{T}$ is zero if and only if $H^n(E) = 0$ for all $n$. A bounded $t$-structure is always non-degenerate. 

The derived category of an abelian category has a canonical non-degenerate $t$-structure whose truncation and cohomology functors are the standard ones, and the categories $D_{\operatorname{Qcoh}}, D^b_{\operatorname{coh}}, D^-_{\operatorname{coh}}$ considered in this paper have canonical non-degenerate $t$-structures compatible with the one on $D(X)$. 

\section{Approximations}
\label{sec:approximations}

Now we are in a position to prove the results of our work. The following two lemmas, though furnished in elementary techniques, are important in applications of interest.

\begin{lemma}\label{lem:samehoms}
    Consider a triangulated category $\mathcal{T}$ equipped with a non-degenerate $t$-structure. If $A \to B$ is a map in $\mathcal{T}$ which is an isomorphism in degrees $\geq a$ for some integer $a$, then the induced map $\operatorname{Hom}(B, E) \to \operatorname{Hom}(A, E)$ is an isomorphism for all $E$ in $\mathcal{T}^{\geq a}$.
\end{lemma}

\begin{proof}
    Consider the following commutative square, and observe the bottom horizontal arrow is an isomorphism from our assumption:
    \begin{displaymath}
        \begin{tikzcd}
        A \ar[r] \ar[d] &B\ar[d] \\
        \tau_{\geq a}A \ar[r] & \tau_{\geq a}B.
        \end{tikzcd}
    \end{displaymath}
    If we apply the functor $\operatorname{Hom}(-, E)$ to the diagram, then the two vertical arrows become isomorphisms, and so the result follows. 
\end{proof}

\begin{lemma}\label{lem:summandofapprox}
    Let $\mathcal{T}$ be a triangulated category equipped with a non-degenerate $t$-structure. Let $a, N$ be integers with $N>0$. Suppose given two morphisms $P \to E$ and $F \to F'$ in $\mathcal{T}$. Assume:
    \begin{itemize}
        \item $P \to E$ is an isomorphism in degrees $\geq a - N$ and $E$ is in $\mathcal{T}^{\geq a}$,
        \item $F \to F^\prime$ is an isomorphism in degrees $\geq a$ and $F^\prime \in \mathcal{T}^{\geq a - N}$. 
    \end{itemize}
    If $P$ is a retract of $F$, then $E$ is a retract of $F^\prime$.
\end{lemma}

\begin{proof}
    There exists a solid diagram where the composition of the horizontal solid arrows is the identity:
    \begin{displaymath}
        \begin{tikzcd}
            P \ar[r] \ar[d] & F \ar[r] \ar[d] & P \ar[d] \\
            E \ar[r, dashed] & F^\prime \ar[r, dashed] & E.
        \end{tikzcd}
    \end{displaymath}
    It follows from \Cref{lem:samehoms} that there exist unique dashed arrows making the diagram commute: The morphisms
    $$
    \operatorname{Hom}(E, F') \to \operatorname{Hom}(P, F') ,\hspace{2em} \operatorname{Hom}(F', E) \to \operatorname{Hom}(F, E)
    $$
    are both isomorphisms. The same argument ensures there exists a unique map $E \to E$ that makes the outer square of the diagram commute. Consequently, we conclude that the composition $E \to F^\prime \to E$ is the identity map of $E$.
\end{proof}

\begin{lemma}\label{lem:approxcoprod}
    Let $\mathcal{T}$ be a triangulated category with a bounded $t$-structure. Choose an object $G$ of $\mathcal{T}$ and positive integer $d$. There exists an $N= N(d) > 0$ such that the following holds: If $a \in \mathbb{Z}$ and $F$ belongs to $\mathcal{A}_d := \operatorname{coprod}_d\big( G(-\infty,\infty)\big)$, then there exist an object $F^\prime$ in $\mathcal{A}_d \cap \mathcal{T}^{\geq a - N}$ and a map $F \to F^\prime$ which is an isomorphism in degrees $\geq a$.
\end{lemma}

\begin{proof}
    We will prove the lemma by induction on $d$ and we will denote $w$ the amplitude of $G$. 
    
    Consider the case  $d=1$. We will prove $N(1) = w$ works. Let $F$ be an object of $\mathcal{A}_1$ and $a$ an integer. Then $F$ is is isomorphic to an object of the form $F^\prime \oplus F^{\prime \prime}$ where $F^{\prime \prime} \in \mathcal{T}^{<a}$, $F^\prime \in \mathcal{T}^{\geq a - w}$ and $F^\prime, F^{\prime \prime} \in \mathcal{A}_1$. The projection $F \to F^\prime$ is the required map for our claim, and so, this furnishes the base case.

    Next, assume the lemma is known for the positive integer $d$. We will prove the lemma for $d+1$ with $N(d+1)=N(d)+w+1$. Let $F \in \mathcal{A}_{d+1}$ and $a \ \in \mathbf{Z}$. Then $F$ is the cone of a map $A \to B$ where $A \in \mathcal{A}_1$ and $B \in \mathcal{A}_d$, see \Cref{rmk:retractsofcoprod}. By our inductive hypothesis, there is a map $B \to B^\prime$ with $B^\prime \in \mathcal{A}_d \cap \mathcal{T}^{\geq a - N(d)}$ which is an isomorphism in degrees $\geq a$. Now as in the base case, $A = A^\prime \oplus A^{\prime \prime}$ where $A^{\prime \prime} \in \mathcal{T}^{< a - N(d)}$, $A^\prime \in \mathcal{T}^{\geq a - N(d)-w}$, and $A^\prime, A^{\prime \prime} \in \mathcal{A}_1$. We claim there is a map $A^\prime \to B^\prime$ fitting into a commutative diagram
    \begin{displaymath}
    \begin{tikzcd}
        A \ar[r] \ar[d] & B \ar[r] \ar[d] & F \ar[r] \ar[d] & A[1] \ar[d] \\
        A^\prime \ar[r] & B^\prime \ar[r] & F^\prime\ar[r] & A^\prime[1]
    \end{tikzcd}
    \end{displaymath}
    where $A \to A^\prime$ is the projection. This follows from \Cref{lem:samehoms}.
    Hence, from this square we obtain the remainder of the desired diagram via the axioms of a triangulated category. Since $A \to A'$ is an isomorphism in degrees $> a$ and $B \to B'$ is an isomorphism in degrees $\geq a$, we see that $F \to F'$ is an isomorphism in degrees $\geq a$, as needed. 
\end{proof}

\begin{proposition}\label{prop:approxisbuilt}
    Let $\mathcal{T}$ be a triangulated category with a bounded $t$-structure. Choose an object $G$ of $\mathcal{T}$ and $d$ a positive integer. There exists an $N>0$ such that the following implies an object $E$ belongs to $\langle G \rangle_d$:
    \begin{enumerate}
        \item $E \in \mathcal{T}^{\geq a}$ for some integer $a$,
        \item $P \to E$ is a map of $\mathcal{T}$ which is an isomorphism in degrees $\geq a - N$,
        \item $P \in \langle G \rangle_d$.
    \end{enumerate}
\end{proposition}

\begin{proof}
    This essentially comes from \Cref{lem:approxcoprod}, and we will keep the notation used there. Suppose (1)--(3) are true for the $N$ found in \Cref{lem:approxcoprod}. By \Cref{rmk:retractsofcoprod}, $P$ is a retract of an object $F$ of $\mathcal{A}_d$. Choose a map $F \to F^\prime$ as in \Cref{lem:approxcoprod} so that we are exactly in the case of \Cref{lem:summandofapprox}. Therefore, $E$ is a retract of $F^\prime$ as desired. 
\end{proof}

\begin{proposition}\label{prop:approximation_bounded_t_structure}
    Let $\mathcal{T}$ be a triangulated category with a bounded $t$-structure. Choose an object $G$ of $\mathcal{T}$ and a positive integer $d$. Assume there exists a collection of objects $\mathcal{P}$ of $\mathcal{T}$ that approximates objects of $\mathcal{T}$ in the following sense: For each object $E$ of $\mathcal{T}$ and integer $a$, there is a map $P \to E$ with $P \in \mathcal{P}$ which is an isomorphism in degrees $\geq a$. Then the following are equivalent: 
    \begin{enumerate}
        \item $\mathcal{T} = \langle G \rangle_d$
        \item $\mathcal{P} \subseteq \langle G \rangle _d$.
    \end{enumerate}
\end{proposition}

\begin{proof}
    It is evident that $(1)\implies (2)$, so we verify the converse direction. Suppose there exist an object $G$ in $\mathcal{T}$ and $d \geq 0$ such that $\mathcal{P}$ is contained in $\langle G \rangle_d$. Let $E$ be any object in $\mathcal{T}$. There exists an integer $a$ such that $E \in \mathcal{T}^{\geq a}$. Take $N$ to be an integer as in \Cref{prop:approxisbuilt}. We can find a map $P \to E$ with $P\in \mathcal{P}$ which is an isomorphism in degrees $\geq a - N$. Since $P$ belongs to $\langle G \rangle_d$, then so does $E$ via \Cref{prop:approxisbuilt}. 
\end{proof}

Let $X$ be a Noetherian scheme. For all $E\in D^-_{\operatorname{coh}}(X)$ and integers $m$, there exists a map $\varphi\colon P\to E$ in $D_{\operatorname{Qcoh}}(X)$ with $P$ perfect on $X$ such that $\varphi$ is an isomorphism in degrees $\geq m$. This is \cite[Theorem 4.1]{Lipman/Neeman:2007}, and is known as \textit{approximation by perfect complexes}.

\begin{proof}[Proof of \Cref{thm:generatingperf}]
    This follows directly from \Cref{prop:approximation_bounded_t_structure} where we take $\mathcal{P}$ to be $\operatorname{perf}X$. See \cite[Theorem 4.1]{Lipman/Neeman:2007}.
\end{proof}

\begin{remark}
    Let $X$ be a Noetherian scheme. If $E$ is an object of $D^b_{\operatorname{coh}}(X)$, then 
    \begin{enumerate}
        \item $\operatorname{Supp}(E):= \cup^\infty_{n=-\infty} \operatorname{Supp}(\mathcal{H}^i (E))$
        \item $E$ is \textit{supported} on a closed subscheme $Z$ of $X$ if $\operatorname{Supp}(E)\subseteq Z$
        \item $E$ is \textit{scheme-theoretically supported} on a closed subscheme $Z$ of $X$ if there exists an object $E^\prime$ of $D^b_{\operatorname{coh}}(Z)$ such that $i_\ast E^\prime \cong E$ in $D^b_{\operatorname{coh}}(X)$ where $i$ is the associated closed immersion.
        \item If $E$ is supported on a closed subscheme $Z$, then there exists a nilpotent thickening $Z \subset Z'$ such that $E$ is scheme-theoretically supported on $Z'$, see \cite[Lemma 7.40]{Rouquier:2008}
    \end{enumerate}
\end{remark}

The following application of \Cref{thm:generatingperf} yields upper bounds on Rouquier dimension via scheme-theoretic support. See the end of \Cref{rmk:retractsofcoprod} for conventions on notation for $\operatorname{add}(-)$.

\begin{theorem}\label{the:rouquier_dimension_via_support_inequality}
    Let $\pi \colon Y \to X$ be a proper morphism of Noetherian schemes. Suppose $U \subset X$ is an open such that the restriction $\pi^{-1}(U) \to U$ is an isomorphism. If the cone of the natural map $\mathcal{O}_X \to \mathbb{R} \pi_\ast \mathcal{O}_Y$ is scheme-theoretically supported on closed subscheme $i \colon Z \to X$ contained in $X\setminus U$, then there is an inequality:
    \begin{displaymath}
        \dim D^b_{\operatorname{coh}}(X) \leq \dim D^b_{\operatorname{coh}}(Y) + \dim D^b_{\operatorname{coh}}(Z) + 1.
    \end{displaymath}
\end{theorem}

\begin{proof}
    If $D^b_{\operatorname{coh}}(Y)$ or $D^b_{\operatorname{coh}}(Z)$ have infinite Rouquier dimension, then the claim holds, so without loss of generality we may assume these values are finite. Let $C$ be the cone of the natural map $\mathcal{O}_X \to \mathbb{R} \pi_\ast \mathcal{O}_Y$. Since $C$ is scheme-theoretically supported on the closed subscheme $i \colon Z \to X$, there exists $C^\prime\in D^b_{\operatorname{coh}}(Z)$ such that $i_\ast C^\prime \cong C$ in $D^b_{\operatorname{coh}}(X)$. There exists a distinguished triangle in $D^b_{\operatorname{coh}}(X)$:
    \begin{displaymath}
        \mathcal{O}_X \to \mathbb{R}\pi_\ast \mathcal{O}_Y \to i_\ast C^\prime \to \mathcal{O}_X[1].
    \end{displaymath}
    If we tensor with a perfect complex $P$ on $X$, then there is another distinguished triangle in $D^b_{\operatorname{coh}}(X)$:
    \begin{displaymath}
        P \to \mathbb{R}\pi_\ast \mathbb{L}\pi^\ast P \to i_\ast (C^\prime \overset{\mathbb{L}}{\otimes} \mathbb{L}i^\ast P) \to P[1].
    \end{displaymath}
    Choose strong generators $G_1, G_2$ respectively for $D^b_{\operatorname{coh}}(Y), D^b_{\operatorname{coh}}(Z)$ with respective generation times $g_1,g_2$. Observe that $\mathbb{R}\pi_\ast G_1$ finitely builds $\mathbb{R}\pi_\ast \mathbb{L}\pi^\ast P$ in at most $g_1 + 1$ cones whereas $i_\ast G_2$ finitely builds $i_\ast (C^\prime \overset{\mathbb{L}}{\otimes} \mathbb{L}i^\ast P)$ in at most $g_2 + 1$ cones. Putting this together, it follows that $\mathbb{R}\pi_\ast G_1 \oplus i_\ast G_2$ finitely builds $P$ in at most $g_1 + g_2 + 2$ cones. Therefore, \Cref{thm:generatingperf} implies the desired claim. 
\end{proof}

\begin{remark}
    \hfill
    \begin{enumerate}
        \item The closed subscheme $Z$ in the
        \Cref{the:rouquier_dimension_via_support_inequality} need not be
        reduced. For example, if $\pi \colon Y \to X$ is the normalisation of an
        integral curve over a perfect field and $Z$ is reduced, then
        \Cref{the:rouquier_dimension_via_support_inequality} gives $\dim
        D^b_{\operatorname{coh}}(X) \leq 2$. However, there exist curves of
        arbitrarily high Rouquier dimension, see \cite[Example
        2.17]{Bai/Cote:2023}. We will see in \Cref{cor:sharp_bound_rouquier_dimension_curves} below a case in which $Z$
        can be taken reduced. 
        \item Note that if $P \in D^b_{\operatorname{coh}}(X)$, then the objects $\mathbb{R}\pi_\ast \mathbb{L}\pi^\ast P$ and  $i_\ast (C^\prime \overset{\mathbb{L}}{\otimes} \mathbb{L}i^\ast P)$ considered in the proof are not necessarily bounded. This is why \Cref{the:rouquier_dimension_via_support_inequality} is useful. Before proving \Cref{thm:generatingperf} we were only able to obtain a significantly worse bound for $\dim D^b_{\operatorname{coh}}(X)$ in terms of $Z$ and $Y$. 
    \end{enumerate} 
\end{remark}

\begin{remark}
    Recall that the \textit{delta invariant} of a one-dimensional reduced Nagata Noetherian ring $A$ is the length of the $A$-module $B/A$, where $B$ is the normalisation of $A$. Thus, $A$ has delta invariant one if, and only if, $B/A \cong A/\mathfrak{m}_A$ as $A$-modules. The delta invariant is one for the local ring of a curve at a nodal or cuspidal singularity, see \cite[Ex. $\S 3.8.1$]{Hartshorne:1983}.
\end{remark}

\begin{corollary}\label{cor:sharp_bound_rouquier_dimension_curves}
    Let $X$ be an excellent reduced purely one-dimensional Noetherian scheme with normalisation $\pi \colon Y \to X$. If the delta invariant of $\mathcal{O}_{X,x}$ is one for every closed point $x$ in $X$, then there is an inequality:
    \begin{displaymath}
        \dim D^b_{\operatorname{coh}}(X) \leq \dim D^b_{\operatorname{coh}}(Y) + 1. 
    \end{displaymath}
\end{corollary}

\begin{proof}
    The hypotheses give us an exact sequence 
    \begin{displaymath}
        0 \to \mathcal{O}_X \to \pi_\ast \mathcal{O}_Y \to \bigoplus _{x \in \operatorname{Sing}(X)} \kappa(x) \to 0. 
    \end{displaymath}
    Since $\pi$ is finite, $\mathbb{R}\pi_\ast \mathcal{O}_Y = \pi_\ast \mathcal{O}_Y$, and the last term can be written $i_\ast \mathcal{O}_{\operatorname{Sing}(X)}$, where $i \colon \operatorname{Sing}(X) \to X$ is the inclusion of the singular locus with the reduced induced closed subscheme structure. We conclude from \Cref{the:rouquier_dimension_via_support_inequality}.
\end{proof}

\begin{remark}
    Following \Cref{cor:sharp_bound_rouquier_dimension_curves}, if either $X$ is affine or $X$ is projective over a field $k$ and $Y$ is smooth (for instance if $k$ is perfect), then it is known that $\dim D^b_{\operatorname{coh}}(Y) = 1$ \cite{Orlov:2009}, and so we get $\dim D^b_{\operatorname{coh}}(X) \leq 2$. This should be compared to \cite[Theorem 10]{Burban/Drozd:2011}, \cite[Corollary 7.2]{Burban/Drozd:2017}, and \cite[Corollary 3.11]{Burban/Drozd:2017a}.
\end{remark}

\begin{lemma}\label{lem:strong_gen_to_strong_oplus_for_coherent_cohomology}
    Let $X$ be an Noetherian scheme. An object $G$ of $D^b_{\operatorname{coh}}(X)$ is a strong generator if, and only if, there exists an $L\geq 0$ such that $D^{-}_{\operatorname{coh}}(X) \subseteq \overline{\langle G \rangle}_L$ in $D_{\operatorname{Qcoh}}(X)$.
\end{lemma}

\begin{proof}
    If $G$ is a strong generator for $D^b_{\operatorname{coh}}(X)$, then there exists an $N\geq 0$ such that $D^b_{\operatorname{coh}}(X) = \langle G \rangle_N$. Given any object $E$ of $D^{-}_{\operatorname{coh}} (X)$, \cite[\href{https://stacks.math.columbia.edu/tag/0DJN}{0DJN}]{StacksProject} tells us $E$ is a homotopy colimit of perfect objects. This gives us a distinguished triangle:
    \begin{displaymath}
        \bigoplus_n P_n \to \bigoplus_n P_n \to E \to (\bigoplus_n P_n)[1]
    \end{displaymath}
    where each $P_n$ is a perfect complex. We know that each $P_n$ is in $\langle G \rangle_N$, and so, $\bigoplus_n P_n$ is in $\overline{\langle G \rangle}_N$. From the distinguished triangle above, it follows that $E\in \overline{\langle G \rangle}_{2N}$. Next, we can check the converse direction. Suppose there exists an $L\geq 0$ such that $D^{-}_{\operatorname{coh}}(X)$ belongs to $\overline{\langle G \rangle}_L$ in $D_{\operatorname{Qcoh}}(X)$. Then $D^b_{\operatorname{coh}}(X)$ is contained in $\overline{\langle G \rangle}_L$, and so it follows that $D^b_{\operatorname{coh}}(X) = \langle G \rangle_L$ via \Cref{rmk:big_to_small_generation_a_la_neeman}.
\end{proof}

\begin{remark}
    Let $X$ be a Noetherian scheme. If $G$ is a strong generator for $D^b_{\operatorname{coh}}(X)$, then the proof of \Cref{lem:strong_gen_to_strong_oplus_for_coherent_cohomology} tells us the following inequality:
    \begin{displaymath}
        \operatorname{gen.time}(G) +1 \leq \min\{ n \geq 0 : D^{-}_{\operatorname{coh}}(X) \subseteq \overline{\langle G \rangle}_n\} \leq 2 (\operatorname{gen.time}(G) +1).
    \end{displaymath}
\end{remark}

The following lemma has significant generalizations which appear elsewhere, see \cite[Theorem 3.11]{Dey/Lank:2024} and for a noncommutative version refer to \cite[Theorem 3.12]{DeDeyn/Lank/ManaliRahul:2024b}. But we record it anyways for its relevance to our work in action.

\begin{lemma}[Aoki]\label{lem:aoki_proper_surjective_descent}
    Let $\pi \colon Y \to X$ be a proper surjective morphism of Noetherian schemes. If $G$ is a strong generator for $D^b_{\operatorname{coh}}(Y)$, then $\mathbb{R}\pi_\ast G$ is a strong generator for $D^b_{\operatorname{coh}}(X)$.
\end{lemma}

\begin{proof}[Proof sketch]
    By \cite[Proposition 4.4]{Aoki:2021}, it can be verified that there exists $k\geq 0$ such that
    \begin{displaymath}
        D^{-}_{\operatorname{coh}}(X) \subseteq \langle \mathbb{R} \pi_\ast \mathcal{O}_Y \overset{\mathbb{L}}{\otimes} D^{-}_{\operatorname{coh}}(X) \rangle_k.
    \end{displaymath}
    Let $E$ be an object of $D^{-}_{\operatorname{coh}} (X)$. By the projection formula \cite[\href{https://stacks.math.columbia.edu/tag/08EU}{Tag 08EU}]{StacksProject}, we have $ \mathbb{R}\pi_\ast \mathcal{O}_Y \overset{\mathbb{L}}{\otimes} E \cong \mathbb{R}\pi_\ast \mathbb{L}\pi^\ast E$. Note that $\mathbb{L}\pi^\ast E$ is in $D^{-}_{\operatorname{coh}}(Y)$. If $G$ is a strong generator for $D^b_{\operatorname{coh}}(Y)$, then \Cref{lem:strong_gen_to_strong_oplus_for_coherent_cohomology} tells us there exists an $n\geq 0$ such that $D^{-}_{\operatorname{coh}}(Y)$ is contained in $\overline{\langle G \rangle}_n$. Hence, we see that $\mathbb{R}\pi_\ast \mathbb{L}\pi^\ast E$ is contained in $\overline{\langle \mathbb{R}\pi_\ast G \rangle}_n$, and so after taking more cones, we see $E$ belongs to $\overline{\langle \mathbb{R}\pi_\ast G \rangle}_{nk}$. \Cref{rmk:big_to_small_generation_a_la_neeman} tells us that $E$ belongs to $\langle \mathbb{R}\pi_\ast G \rangle_{nk}$. In other words, $D^b_{\operatorname{coh}}(X)$ is contained in $\langle \mathbb{R}\pi_\ast G \rangle_{nk}$, which completes the proof.
\end{proof}

\begin{proposition}\label{prop:proper_birational}
    If $\pi \colon Y \to X$ is a proper surjective morphism of Noetherian schemes, then the Rouquier dimension of $D^b_{\operatorname{coh}}(X)$ is bounded above by
    \begin{displaymath}
        (\dim D^b_{\operatorname{coh}}(Y )+ 1) \cdot \min\{ \operatorname{level}_{D(X)}^{\mathbb{R}\pi_\ast G} (\mathcal{O}_X) : G \in D^b_{\operatorname{coh}}(Y) \} -1.
    \end{displaymath}
\end{proposition}
    
\begin{proof}
    By \Cref{lem:aoki_proper_surjective_descent}, if $G$ is a strong generator for $D^b_{\operatorname{coh}}(Y)$, then $\mathbb{R}\pi_\ast G$ is a strong generator for $D^b_{\operatorname{coh}}(X)$. Choose a perfect complex $P$ on $X$. If $n$ is the level of $\mathcal{O}_X$ with respect to $\mathbb{R}\pi_\ast G$ in $D^b_{\operatorname{coh}}(X)$, then tensoring with $P$ tells us $P$ belongs to $\langle \mathbb{R}\pi_\ast G \overset{\mathbb{L}}{\otimes} P \rangle_n$. However, projection formula ensures that $\mathbb{R}\pi_\ast G \overset{\mathbb{L}}{\otimes} P$ is isomorphic to $\mathbb{R}\pi_\ast (G \overset{\mathbb{L}}{\otimes} \mathbb{L}\pi^\ast P)$. Suppose the generation time of $G$ in $D^b_{\operatorname{coh}}(Y)$ is $g$. Then $G \overset{\mathbb{L}}{\otimes} \mathbb{L}\pi^\ast P$ belongs to $\langle G \rangle_{g+1}$, and so, $\mathbb{R}\pi_\ast G \overset{\mathbb{L}}{\otimes} P$ belongs to $\langle \mathbb{R}\pi_\ast G \rangle_{g+1}$. Tying this together, we see that $P$ belongs to $\langle \mathbb{R}\pi_\ast G\rangle_{(g+1)n}$, and so, the Rouquier dimension of $D^b_{\operatorname{coh}}(X)$ is bounded above by $(g+ 1) n -1$ by \Cref{thm:generatingperf}. This completes the proof.
\end{proof}

The following can be seen as an application of \Cref{thm:summandofpullpush} or \Cref{prop:proper_birational}.

\begin{example}\label{ex:derived_splinters}
    Recall that a Noetherian scheme $X$ is a \textbf{derived splinter} (respectively \textbf{birational derived splinter}) if, for every proper (respectively birational) surjective morphism $\pi\colon Y \to X$, the natural map $\mathcal{O}_X \to \mathbb{R}\pi_\ast \mathcal{O}_Y$ splits in the derived category $D_{\operatorname{Qcoh}}(X)$. For $X$ a variety over the complex numbers, this is equivalent to $X$ having rational singularities \cite{Kovacs:2000, Bhatt:2012}. However, for $X$ of prime characteristic, rational singularities are distinct from being a derived splinter, as seen in \cite[Example 2.11]{Bhatt:2012}. For additional background information, please refer to \cite{Bhatt:2012, Lyu:2022}. If $X$ is (birational) derived splinter variety over a field and $X$ admits a resolution of singularities $\widetilde{X} \to X$, \Cref{prop:proper_birational} asserts that the Rouquier dimension of $D^b_{\operatorname{coh}}(X)$ is at most the Rouquier dimension of $D^b_{\operatorname{coh}}(\widetilde{X})$, which is at most $2 \dim X$ by \cite[Proposition 7.9]{Rouquier:2008}. Notably, if \cite[Conjecture 10]{Orlov:2009} holds and $X$ admits a resolution of singularities, then the Rouquier dimension of $D^b_{\operatorname{coh}}(X)$ is precisely $\dim X$. On a different topic, see \cite{Lank/Venkatesh:2024} for a triangulated characterization of derived splinters in terms of generation.
\end{example}

In the proof of part (1) of \Cref{thm:summandofpullpush} we will need the following result of Balmer.

\begin{theorem}[{\cite{Balmer:2016}}]\label{thm:balmer}
    If $f \colon  U \to X$ is a separated \'{e}tale morphism of quasi-compact and quasi-separated schemes, then every perfect complex on $U$ is a direct summand of $\mathbb{L} f^\ast$ of a perfect complex on $X$. 
\end{theorem}

We give an alternate proof based on the following two lemmas.

\begin{lemma} \label{lem:hocolimperf}
    Let $f \colon U \to X$ be a finitely presented, flat morphism of schemes with $X$ quasi-compact and quasi-separated. Then for any perfect complex $P$ on $U$ there exist perfect complexes $P_n$ on $X$ such that
    \begin{displaymath}
        \mathbb{R} f_\ast P = \operatorname{hocolim}_n P_n.
    \end{displaymath}
\end{lemma}

\begin{proof}[Proof sketch]
    By absolute Noetherian approximation \cite[Appendix C]{Thomason/Trobaugh:1990} and flat base change, this reduces to the case when $X$ is of finite type over $\mathbb{Z}$. In this case the result follows from \cite[\href{https://stacks.math.columbia.edu/tag/0CRM}{Tag 0CRM}]{StacksProject} via the equivalence \cite[\href{https://stacks.math.columbia.edu/tag/09M5}{Tag 09M5}]{StacksProject}.
\end{proof}

\begin{lemma}\label{lem:summandofsomething}(\cite[proof of Theorem 3.5]{Balmer:2016})
    Let $f \colon U \to X$ be a quasi-compact and separated \'{e}tale morphism of schemes. Then every object $K \in D_{\operatorname{Qcoh}}(X)$ is a direct summand of the object $\mathbb{L}f^\ast \mathbb{R} f_\ast K$.
\end{lemma}

\begin{proof}
    The diagonal $\Delta \colon U \to U \times _X U$ is both an open and closed immersion. Thus $\mathbb{R}\Delta_\ast \mathcal{O}_U$ is a direct summand
    of $\mathcal{O}_{U \times _X U}$. But then $\mathbb{R} \operatorname{pr}_{2,\ast}( \mathbb{L} \operatorname{pr}_{1}^\ast K \otimes^{\mathbb{L}} \mathbb{R}\Delta_\ast \mathcal{O}_U)$ 
    is a direct summand of $\mathbb{R}\operatorname{pr}_{2,\ast} (\mathbb{L}\operatorname{pr}_{1}^\ast K \otimes^{\mathbb{L}} \mathcal{O}_{U \times _X U})$. The first object is $K$ and the second object is 
    \begin{displaymath}
        \mathbb{R} \operatorname{pr}_{2,\ast} \mathbb{L} \operatorname{pr}_{1}^\ast K = \mathbb{L}f^\ast \mathbb{R}f_\ast K
    \end{displaymath}
    by flat base change.
\end{proof}

\begin{proof}[Proof of \Cref{thm:balmer}]
    By \Cref{lem:summandofsomething} there are maps $E \xrightarrow{i} \mathbb{L}f^\ast \mathbb{R}f_\ast E \xrightarrow{p} E$ whose composition is the identity. By \Cref{lem:hocolimperf} the middle object can be written $\operatorname{hocolim}_n \mathbb{L}f^\ast P_n$ with $P_n$ a perfect complex on $X$. But since $E$ is perfect, the map $i$ factors through some $\mathbb{L}f^\ast P_n$ and we are done.
\end{proof}

Now we are in a position to prove \Cref{thm:summandofpullpush}. We choose to give direct arguments, but we leave it to the reader to check that it can also be proven by applying \Cref{lem:summandofapprox}.

\begin{proof}[Proof of \Cref{thm:summandofpullpush}]
    (1) Choose an object $E$ in $D^b_{\operatorname{coh}}(U)$. By \cite[Theorem 4.1]{Lipman/Neeman:2007}, there exist a perfect complex $P$ on $U$ and an integer $a$ such that $E = \tau_{\geq a}(P)$. Moreover, \Cref{thm:balmer} guarantees there exists another perfect complex $Q$ on $X$ such that $P$ is a direct summand of $\mathbb{L} f^\ast Q$. Since $\tau_{\geq a}$ is a functor, we see that $E = \tau_{\geq a}(P)$ is a direct summand of $\tau_{\geq a}(\mathbb{L} f^\ast Q) = \mathbb{L} f^\ast(\tau_{\geq a}Q) $, and so we are done.

    (2) Let $E \in D^b_{\operatorname{coh}}(X)$. Assume $E \in D^{\geq a}(X)$. Choose an integer $N>0$ such that $\mathbb{R}f_\ast (D^{\leq 0}_{\operatorname{coh}}(Y)) \subset D^{\leq N}_{\operatorname{coh}}(X)$. 
    Set $K = \tau_{\geq a - N} \mathbb{L}f^\ast E \in D^b_{\operatorname{coh}}(Y)$. Write $F = \mathbb{R}f_\ast \mathbb{L}f^\ast E$ and let $F\to F^\prime$ be the obvious map $F \to \mathbb{R}f_\ast K$. Then $F \to F^\prime$ is an isomorphism in degrees $\geq a$, so by \Cref{lem:samehoms} there is a unique map $F^\prime \to E$ making the triangle
    \begin{displaymath}
        \begin{tikzcd}
            F \ar[r] \ar[d] &E \\
            F^\prime \ar[ur]
        \end{tikzcd}
    \end{displaymath}
    commute, where $F \to E$ is any splitting of $E \to F$. One verifies easily that the composition $E \to F \to F^\prime \to E$ is the identity.
\end{proof}

Next up, we study the behavior of Rouquier dimension along \'{e}tale morphisms with affine target, and in particular, prove \Cref{thm:etale_descent_rouquier_dimension}. The subsequent two statements will use the following setup.

\begin{setup}\label{setup:etale_setup}
    Let $f \colon U \to X$ be an \'{e}tale morphism of Noetherian affine schemes which admits a factorization $U \xrightarrow{j} \overline{U} \xrightarrow{g} X$ such that $g$ is finite and $j$ is a principal open immersion, that is, the inclusion of an open subscheme of the form $D(b)$ for some $b \in H^0(\overline{U}, \mathcal{O}_{\overline{U}})$. Let $G$ be an object of $D^b_{\operatorname{coh}}(U)$. Choose an object $\overline{G}$ in $D^b_{\operatorname{coh}}(\overline{U})$ such that $\mathbb{L}j^\ast \overline{G} = G$. Consider the following diagram:
    \begin{displaymath}
        \cdots \xrightarrow{b} \overline{G} \xrightarrow{b} \overline{G} \xrightarrow{b} \cdots \xrightarrow{b} \overline{G} 
    \end{displaymath}
    in $D^b_{\operatorname{coh}}(\overline{U})$. Applying the functor $\mathbb{R}g_\ast$ gives a diagram in $D^b_{\operatorname{coh}}(X)$:
    \begin{equation}
        \label{eq:delignediagram}
        \cdots \to \mathbb{R}g_\ast \overline{G} \to \mathbb{R}g_\ast \overline{G} \to \cdots \to \mathbb{R}g_\ast \overline{G}
    \end{equation}
\end{setup}

Deligne proves in \cite[Appendix by P. Deligne]{Hartshorne:1966} that the assignment $G \mapsto $ the pro-system (\ref{eq:delignediagram}) is adjoint to the functor $\mathbb{L}f^\ast$ in the following precise sense. 
 
\begin{lemma}[Deligne's formula]\label{lem:deligne}
    Assume \Cref{setup:etale_setup}. For an object $K$ of $D^b_{\operatorname{coh}}( X)$, there is a canonical isomorphism
    \begin{displaymath}
        \operatorname{colim}_n \operatorname{Hom}_{D^b_{\operatorname{coh}}( X)} (\mathbb{R}g_\ast \overline{G}, K) = \operatorname{Hom}_{D^b_{\operatorname{coh}}( U)}(G, \mathbb{L}f^\ast K),
    \end{displaymath}
    where the colimit is over Diagram~\ref{eq:delignediagram}. This isomorphism is defined as follows. The $n^{th}$ map 
    \begin{displaymath}
        \operatorname{Hom}_{D^b_{\operatorname{coh}}( X)} (\mathbb{R}g_\ast \overline{G}, K) \to \operatorname{Hom}_{D^b_{\operatorname{coh}}( U)}(G, \mathbb{L}f^\ast K) 
    \end{displaymath}
    takes $\varphi$ to the composition
    \begin{equation}
    \label{eq:deligneprecise}
        G \xrightarrow{1/b^n} G \to \mathbb{L}f^\ast \mathbb{R}g_\ast \overline{G} \xrightarrow{\mathbb{L}f^\ast(\varphi)} \mathbb{L}f^\ast K,
    \end{equation}
    where $G \to \mathbb{L}f^\ast \mathbb{R}g_\ast \overline{G}$ is $\mathbb{L}j^\ast$ applied to the unit map $\overline{G} \to \mathbb{L}g^\ast \mathbb{R}g_\ast \overline{G}$. 
\end{lemma}

\begin{proof}[Proof sketch]
    The key point is that for an \'{e}tale morphism, $f^! = \mathbb{L}f^\ast$. This means that there is a canonical isomorphism
    \begin{displaymath}
        \mathbb{L}f^\ast K = \mathbb{L}j^\ast (g^!(K))
    \end{displaymath}
    valid for $K \in D^b_{\operatorname{coh}}(X)$, see \cite[\href{https://stacks.math.columbia.edu/tag/0DWE}{Tag 0DWE}]{StacksProject}. Further, we have for any object $L \in D_{\operatorname{Qcoh}}(\overline{U})$,
    \begin{displaymath}
        \mathbb{R}j_\ast \mathbb{L}j^\ast L = L \overset{\mathbb{L}}{\otimes} \mathbb{R}j_\ast \mathcal{O}_U = L \overset{\mathbb{L}}{\otimes} (\operatorname{hocolim} (\mathcal{O}_{\overline{U}} \xrightarrow{b} \mathcal{O}_{\overline{U}} \xrightarrow{b} \cdots))
        = \operatorname{hocolim} (L \xrightarrow{b} L \xrightarrow{b} \cdots),
    \end{displaymath}
    and thus for any perfect complex $P$ on $\overline{U}$,
    \begin{displaymath}
       \operatorname{Hom}_{D(\mathcal{O}_U)} (\mathbb{L}j^\ast P, \mathbb{L}j^\ast L) = \operatorname{Hom}(P, \mathbb{R}j_\ast \mathbb{L}j^\ast L ) = \operatorname{colim}_n \operatorname{Hom}(P, L)
    \end{displaymath}
    (colimit taken over the system $L \xrightarrow{b} L \xrightarrow{b} \cdots$). If $L \in D^{+}_{\operatorname{Qcoh}}(\overline{U})$ then this formula is also valid for $P \in D^b_{\operatorname{coh}}(\overline{U}),$ as follows by approximating $P$ suitably by a perfect complex. Putting everything together, we get
    \begin{displaymath}
        \begin{aligned}
            \operatorname{colim}_n \operatorname{Hom}_{D^b_{\operatorname{coh}}( X)} (\mathbb{R}g_\ast \overline{G}, K) &= \operatorname{colim}_n \operatorname{Hom}_{D^b_{\operatorname{coh}}( \overline{U})} (\overline{G}, g^!(K)) 
            \\&= \operatorname{Hom}_{D^b_{\operatorname{coh}}( U)} (G, \mathbb{L}j^\ast(g^!(K))) \\&=  \operatorname{Hom}_{D^b_{\operatorname{coh}}( U)} (G, \mathbb{L}f^\ast K),
        \end{aligned}
    \end{displaymath}
    where the first equality is duality for a finite morphism.
\end{proof}

\begin{lemma}\label{lem:ghostlocaletale}
    Assume \Cref{setup:etale_setup}. If $\psi \colon K \to L$ is an $\mathbb{R}g_\ast \overline{G}$-ghost map in $D^b_{\operatorname{coh}}( X)$, then $\mathbb{L}f^\ast(\psi)$ is a $G$-ghost map in $D^b_{\operatorname{coh}}( U)$. 
\end{lemma}

\begin{proof}
    Assume $\mathbb{L}f^\ast(\psi) \colon \mathbb{L}f^\ast K \to
    \mathbb{L}f^\ast L$ is not a $G$-ghost map. After replacing, if necessary,
    $\psi$ by some shift $\psi[i]$, this means there is a nonzero composition $G
    \to \mathbb{L}f^\ast K\to \mathbb{L}f^\ast L$. By \Cref{lem:deligne},
    there exist an integer $n$ and a map $\varphi \colon \mathbb{R}
    g_\ast \overline{G} \to K$ such that $G \to \mathbb{L}f^\ast K$ is equal to
    a composition of the form in Diagram~\ref{eq:deligneprecise}. But then from
    this diagram and the fact that $G \to \mathbb{L}f^\ast K \to
    \mathbb{L}f^\ast L$ is nonzero, we see that $\mathbb{L}f^\ast(\psi) \circ
    \mathbb{L}f^\ast(\varphi) \neq 0$, and hence $\psi \circ \varphi \neq 0$.
    This tells us $\psi$ is not a $\mathbb{R}g_\ast \overline{G}$-ghost map.
\end{proof}

\begin{lemma}\label{lem:ghostaffinepullback}
    Let $f\colon U \to X$ be a flat morphism of Noetherian affine schemes. Let $G, K, L \in D^b_{\operatorname{coh}}(X)$ be objects and $\psi \colon K \to L$ a map. Then:
    \begin{enumerate}
        \item If $\psi$ is a $G$-ghost, then $\mathbb{L}f^\ast(\psi)$ is an $\mathbb{L}f^\ast(G)$-ghost; and the converse holds if $f$ is surjective.
        \item If $f$ is surjective and $\mathbb{L}f^\ast(\psi) = 0$, then $\psi = 0$. 
    \end{enumerate}
\end{lemma}

\begin{proof}
    Let $f$ correspond to the ring map $R \to A$. Then everything follows from the fact that for $E,F \in D^b_{\operatorname{coh}}(X),$ the $R$-module $\operatorname{Hom}(E, F)$ is finitely generated and satisfies
    \begin{displaymath}
        \operatorname{Hom}_{D^b_{\operatorname{coh}}(X)}(E, F) \otimes_R A = \operatorname{Hom}_{D^b_{\operatorname{coh}}(U)}(\mathbb{L}f^\ast E, \mathbb{L}f^\ast F),
    \end{displaymath}
    see \cite[\href{https://stacks.math.columbia.edu/tag/0A6A}{Tag 0A6A}]{StacksProject}.
\end{proof}

\begin{lemma}\label{lem:ghostetale}
    Let $f \colon U \to X$ be an \'{e}tale morphism of Noetherian affine schemes. Given an object $G$ of $D^b_{\operatorname{coh}}( U)$, there exists an object $H$ of $D^b_{\operatorname{coh}}( X)$ with the following property: If $\psi \colon K \to L$ is an $H$-ghost map in $D^b_{\operatorname{coh}}( X)$, then $\mathbb{L}f^\ast(\psi)$ is a $G$-ghost map in $D^b_{\operatorname{coh}}(U)$.
\end{lemma}

\begin{proof}
    By Zariski's Main Theorem \cite[IV Corollaire 18.12.13]{EGA}, we may factor $f$ as an open immersion $U \to \overline{U}$ followed by a finite morphism $\overline{U} \to X$. Then $\overline{U}$ is affine so there exists a finite open covering $U = \bigcup _i U_i$ such that each $U_i \to \overline{U}$ is a principal open. By \Cref{lem:ghostlocaletale}, our claim is true for the maps $U_i \to X$ so there are objects $H_i \in D^b_{\operatorname{coh}}( X)$ such that $H_i$-ghosts pullback to $G|_{U_i}$-ghosts along $U_i \to X$. From \Cref{lem:ghostaffinepullback} applied to the covering $\coprod_i U_i \to U$, it follows that $H = \bigoplus _i H_i$ fits the bill.
\end{proof}

\begin{proof}[Proof of \Cref{thm:etale_descent_rouquier_dimension}]
    It follows from \Cref{thm:summandofpullpush}(1) that $\dim D^b_{\operatorname{coh}}(X)\geq \dim D^b_{\operatorname{coh}}(U)$, so we check the reverse inequality. Let $d = \operatorname{dim} D^b_{\operatorname{coh}}(U)$ and choose an object $G$ such that $\langle G \rangle _{d+1} = D^b_{\operatorname{coh}}(U)$. Let $H \in D^b_{\operatorname{coh}}(X)$ be as in \Cref{lem:ghostetale}. We claim $\langle H \rangle _{d+1} = D^b_{\operatorname{coh}}(X)$. Our assumption on dualizing complexes lets us appeal to Letz's converse ghost lemma \cite[$\S 2.13$]{Letz:2021}, and so, from \Cref{rmk:coghost_lemma}, it suffices to show a composition of $d+1$ $H$-ghosts in $D^b_{\operatorname{coh}}(X)$ vanishes. By \Cref{lem:ghostetale} and the ghost lemma \cite[$\S 2.13$]{Letz:2021}, the pullback of such a composition to $U$ vanishes. But then by part (2) of \Cref{lem:ghostaffinepullback}, the composition itself vanishes, completing the proof.
\end{proof}

\begin{lemma}
\label{lem:descentforcoh}
    Assume $X = \operatorname{lim}_i X_i$ is a co-filtered limit with flat, affine transition maps of schemes $X_i$. Assume all the $X_i$ and $X$ are Noetherian. Then the category $D^b_{\operatorname{coh}}(X)$ is the colimit of the categories $D^b_{\operatorname{coh}}(X_i)$. 
\end{lemma}

Here ``colimit" means 2-colimit in the 2-category of categories. 
More explicitly, the Lemma says that every object of 
$D^b_{\operatorname{coh}}(X)$ is isomorphic to the derived pullback 
of an object of $D^b_{\operatorname{coh}}(X_i)$ for some $i$, and if 
$K_i, L_i \in D^b_{\operatorname{coh}}(X)$, then
$$
\operatorname{Hom}_{D^b_{\operatorname{coh}}(X)}(K, L) = \operatorname{colim}_{j \geq i}\operatorname{Hom}_{D^b_{\operatorname{coh}}(X_j)} (K_j, L_j),
$$
where $K, L$ (resp. $K_j, L_j$) are the pullbacks of $K_i, L_i$ to $X$ (resp. $X_j$).

\begin{proof}
    Derived pullback along the morphisms $X \to X_i$ induces a functor from the colimit category to $D^b_{\operatorname{coh}}(X)$. To show the functor is essentially surjective, we have to show every $E \in D^b_{\operatorname{coh}}(X)$ is a pullback from $X_i$ for $i \gg  0$. We may write $E = \tau_{\geq a}P$ for some perfect complex $P$ on $X$ and integer $a$. Then by \cite[\href{https://stacks.math.columbia.edu/tag/09RF}{Tag 09RF}]{StacksProject}, $P$ is the pullback of a perfect complex $P_i$ on $X_i$ for $i \gg 0$. Then since $X \to X_i$ is flat, $E$ is the pullback of the complex $\tau_{\geq a}P_i$ on $X_i$. Fully faithfulness follows from \cite[\href{https://stacks.math.columbia.edu/tag/09RE}{Tag 09RE}]{StacksProject}.
\end{proof}

\begin{corollary}
\label{cor:hensel}
    Let $R$ Noetherian local ring and let $S$ denote its Henselization or strict Henselization. Then 
    \begin{displaymath}
        \operatorname{dim}(D^b_{\operatorname{coh}}(S)) \leq \operatorname{dim}(D^b_{\operatorname{coh}}(R)),
    \end{displaymath}
    and equality holds if $R$ admits a dualizing complex. 
\end{corollary}

\begin{proof}
    Note that Henselization of our Noetherian local ring remains Noetherian, see \cite[\href{https://stacks.math.columbia.edu/tag/06LJ}{Tag 06LJ}]{StacksProject}. We will show this more generally for any local homomorphism $R \to S$ of Noetherian local rings which is a filtered colimit $S = \operatorname{colim}_i S_i$ of \'etale ring maps $R \to S_i$. To show the inequality, we will show every $K \in D^b_{\operatorname{coh}}(S)$ is a direct summand of $L \otimes _R ^{\mathbb{L}} S$ for some object $L \in D^b_{\operatorname{coh}}(R)$. So let $K \in D^b_{\operatorname{coh}}(S)$. Then by \Cref{lem:descentforcoh}, there exists an $i$ and on object $K_i \in D^b_{\operatorname{coh}}(S_i)$ such that $K_i \otimes _{S_i} ^{\mathbb{L}} S \cong K$, and then by \Cref{thm:summandofpullpush} there is an object $L \in D^b_{\operatorname{coh}}(R)$ such that $K_i$ is a direct summand of $L \otimes _R ^{\mathbb{L}} S_i$. But then $K$ is a direct summand of $L \otimes _R ^{\mathbb{L}} S$.

    Now assume $R$ admits a dualizing complex. This case allows us to appeal to Letz's converse ghost lemma \cite[$\S 2.13$]{Letz:2021} and \Cref{rmk:coghost_lemma}. We will show the reverse inequality. Let $d = \operatorname{dim}(D^b_{\operatorname{coh}}(S))$ and choose $G \in D^b_{\operatorname{coh}}(S)$ such that $D^b_{\operatorname{coh}}(S) = \langle G \rangle _{d+1}$. Then by \Cref{lem:descentforcoh} again, there exists an $i$ and an object $G_i \in D^b_{\operatorname{coh}}(S_i)$ such that $G_i \otimes _{S_i} ^{\mathbb{L}} S  \cong G$. By \Cref{lem:ghostetale}, there exists an object $H \in D^b_{\operatorname{coh}}(R)$ such the base change of an $H$-ghost along $R \to S_i$ is a $G_i$-ghost. But then by \Cref{lem:ghostaffinepullback}, the base change of an $H$-ghost along $R \to S$ is a $G$-ghost, and we conclude as in the proof of \Cref{thm:etale_descent_rouquier_dimension}.
\end{proof}

\begin{corollary} 
\label{cor:nodal}
    Let $X$ be an affine, purely one-dimensional reduced scheme of finite type over an algebraically closed field $k$. Assume the singularities of $X$ are at worst nodal. Then $\operatorname{dim}D^b_{\operatorname{coh}}(X) = 1$. 
\end{corollary}

\begin{proof}
    We have $\operatorname{dim}D^b_{\operatorname{coh}}(X) \geq 1$ by \cite[Proposition 7.16]{Rouquier:2008}. 
    Now for every singular point $x \in X$, there are an affine scheme $U$ with a closed point $u \in U$ and \'{e}tale morphisms $U \to X$ and $U \to C$, where $C = V(T_0T_1) \subset \mathbf{P}^2_k$, which take $u \mapsto x$ and $u \mapsto (0 : 0 : 1)$. This follows from Artin approximation \cite[Corollary 2.6]{Artin-Algebraic-Approximation} and the definition of a node. Thus $X$ has an \'etale covering by the disjoint union of a regular one dimensional affine scheme and schemes of the form $U$ above, so it suffices by \Cref{thm:etale_descent_rouquier_dimension} to show $U$ has Rouquier dimension $\leq 1$ (this is known for a regular one dimensional affine scheme). But by \cite[Corollary 3.11]{Burban/Drozd:2017}, the scheme $C$ has Rouquier dimension 1, hence $U$ has Rouquier dimension $\leq 1$ by \Cref{thm:summandofpullpush}.
\end{proof}

\bibliographystyle{alpha}
\bibliography{mainbib}

\end{document}